\documentclass{amsart}

\usepackage[bookmarksnumbered,colorlinks, plainpages]{hyperref}
\hypersetup{colorlinks=true,linkcolor=red, anchorcolor=green, citecolor=cyan, urlcolor=red, filecolor=magenta, pdftoolbar=true}

\usepackage{amsmath, amsthm, amscd, amsfonts, amssymb, graphicx, color, enumerate, xcolor, mathrsfs, latexsym}
\usepackage{longtable}

\theoremstyle{plain} 
\newtheorem{theorem}{Theorem}[section]

\theoremstyle{definition}

\theoremstyle{remark}

\renewcommand{\leq}{\leqslant} 
\renewcommand{\geq}{\geqslant}

\newcommand{\gen}[1]{\langle#1\rangle}

\newcommand{\M}{\mathcal{M}}

\begin{document}
	\title{On the pair of nilpotent Lie algebras}
	\author[M. Sajedi]{Mostafa Sajedi$^{1}$}
	\author[M. R. R. Moghaddam]{Mohammad Reza R. Moghaddam$^{2}$}
	
	\date{}
	
	\address{$^{1}$ Department of Mathematics, Mashhad Branch, Islamic Azad University, Mashhad, Iran}	
	\address{$^{2}$ Department of Mathematics, Khayyam University, Mashhad, Iran,\\
		Department of Pure Mathematics, Centre of Excellence in Analysis on Algebraic Structures (CEAAS), Ferdowsi University of Mashhad, P.O. Box 1159, Mashhad, Iran and\\
		}

	\email{m.sajedi@esfarayen.ac.ir}
	\email{m.r.moghaddam@khayyam.ac.ir; rezam@ferdowsi.um.ac.ir}	

\begin{abstract}
	
     Let $(N,L)$ be a pair of finite dimensional nilpotent Lie algebras and $N$ admits a complement $K$ in $L$ such that $\dim N=n$ and $\dim K=m$. Let $s(N,L)=\frac{1}{2}(n-1)(n-2)+1+(n-1)m-\dim \M(N,L)$, where $\M(N,L)$ denotes the multiplier of the pair $(N,L)$. The aim of this paper is to classify all pairs of finite dimensional nilpotent Lie algebras $(N,L)$, for which $s(N,L)=6,7$.

\end{abstract}
\maketitle
\section{Introduction ad preliminaries}

Suppose $L$ is a Lie algebra over an algebraically closed field $\mathcal{K}$ with $char(\mathcal{K})\neq 2$ and $L=F/R$ is a free presentation of $L$, where $F$ is a free Lie algebra. Then we call $(R\cap F^2)/[R,F]$ the multiplier of $L$ and denote it by $\M(L)$. Now, let $N$ be an ideal of $L$, where $N\cong S/R$. Therefore, we define the multiplier of the pair $(N,L)$ to be $\M (N,L)=(R\cap[S,F])/[R,F]$. One notes that if $N=L$, then $\M (N,L)\cong \M (L)$.

We associate to every $n$-dimensional nilpotent Lie algebra two variants $t(L)$ and $s(L)$ as follows
\begin{gather*}
	t(L)=\frac{1}{2}n(n-1)-\dim\mathcal{M}(L),\\
	s(L)=\frac{1}{2}(n-1)(n-2)+1-\dim\mathcal{M}(L).
\end{gather*}
(For more detailes one may refer to \cite{hd.fs.me.1}, \cite{zh.fs.hd2} and  \cite{pn}). 

Nilpotent Lie algebras with $t(L)\le 10$ and $s(L) \leq 7$ are classified in \cite{pb.km.es,hd.me.ha, hd.fs.me.2, ph, ph.es, pn.as, as.pn.1, as.pn.2}.  Also, in  \cite{hd.me.bj}, filiform Lie algebras for $t(L)\le 23$ are classified. 

Khamseh and Alizadeh Niri \cite{ek.sa}, define two non negative variants $t(N,L)$ and $s(N,L)$ for the pair of Lie algebras $(N,L)$ as follows
\begin{gather*}
	t(N,L)=\frac{1}{2}n(n+2m-1)-\dim\mathcal{M}(N,L),\\
	s(N,L)=\frac{1}{2}(n-1)(n-2)+1+(n-1)m-\dim\mathcal{M}(N,L),
\end{gather*}
where $m=\dim{K}$ and $K$ is the complement of $N$ in $L$. They also classify all pairs of Lie algebras $(N,L)$, when $t(N,L)\le 6$ and $s(N,L)\le 2$.

In the present paper we classify all pairs of Lie algebras $(N,L)$, when $s(N,L)=6,7$. With a similar argument one can classify all pairs of Lie algebras $(N,L)$, when $s(N,L)\le 5$, which we leave to the reader. Finally, note that in the present paper we substantially use the ideas of \cite{hd.me.ha} and \cite{hd.me.bj} to do our classifications.
In this paper, $H(m)$ denotes the Heisenberg Lie algebra of dimension $2m+1$ and is given by
$H(m)=\gen{x,x_1,\ldots,x_{2m}:[x_{2i-1},x_{2i}]=x,i=1,\ldots,m}.$
Moreover, $A(n)$ is an $n$-dimensional abelian Lie algebras.

A Lie algebra $L$ is called a central product of ideals $M$ and $N$, if $L = M + N$, $[M, N] = 0$ and $M \cap N \subseteq Z(L)$. We denote the central product of two Lie algebras $M$ and $N$ by $M \dotplus N$.
In the following theorems, the structure of nilpotent Lie algebras $L$ with $s(L) \leq 7$ is given.
\begin{theorem}[\cite{pn,hd.fs.me.2}] \label{s=0,1,2,3}
	Suppose $L$ is an $n$-dimensional non-abelian nilpotent Lie algebra. Then
	\begin{itemize}
		\item[(1)]$s(L)=0$ if and only if $L\cong H(1)\oplus A(n-3)$.
		\item[(2)]$s(L)=1$ if and only if $L\cong L_{5,8}$.
		\item[(3)]$s(L)=2$ if and only if $L$ is isomorphic to $L_{4,3}$,$L_{5,8}\oplus A(1)$,or $H(r)\oplus A(n-2r-1),$
		for some $r\geq2$.
		\item[(4)]$s(L)=3$ if and only if $L$ is isomorphic to $L_{4,3}\oplus A(1)$, $L_{5,5}$, $L_{6,22}(\epsilon)$, $L_{6,26}$ or $L_{5,8}\oplus A(2)$.
	\end{itemize}
\end{theorem}

\begin{theorem}[\cite{as.pn.1,as.pn.2}] \label{s=4,5}
	Let $L$ be a non-abelian $d$-dimensional nilpotent Lie algebra. Then	
	\begin{itemize}
		\item[(i)]$s(L)=4$ if and only if $L$ is isomorphic to one of the Lie algebras 
		$L_{5,8}\oplus A(3)$, $L_{4,3}\oplus A(2)$,	$L_{5,5}\oplus A(1)$, $L_{5,6}$, $L_{5,7}$,  $L_{5,9}$,$L_{6,22}(\epsilon)\oplus A(1)$ or $37A$.
				
		\item[(ii)] $s(L)=5$ if and only if $L$ is isomorphic to one of the Lie algebras 
		$L_{5,8}\oplus A(4)$, $L_{4,3}\oplus A(3)$,	$L_{5,5}\oplus A(2)$, $L_{6,22}(\epsilon)\oplus A(2)$,$L_{6,26}\oplus A(1)$, $L_{6,10}$,  $L_{6,23}$, $L_{6,25}$, $L_{6,27}$, $37B$ or $37D$.
		
	\end{itemize}	
\end{theorem}

\begin{theorem}[\cite{pn.as}] \label{s=6,7}
	Let $L$ be a non-abelian $d$-dimensional nilpotent Lie algebra. Then	
	\begin{itemize}
		\item[(i)]$s(L)=6$ if and only if $L$ is isomorphic to one of the Lie algebras
			$L_{5,8}\oplus A(5)$, $L_{4,3}\oplus A(4)$,	$L_{5,5}\oplus A(3)$, $L_{6,22}(\epsilon)\oplus A(3)$,
		$L_{6,10}\oplus A(1)$,  $27A$,  $157$,  $37A \oplus A(1)$, $L_{6,6}$,  $L_{6,7}$,  $L_{6,9}$, $L_{6,11}$, $L_{6,12}$,	$L_{6,19}(\epsilon)$, $L_{6,20}$ or $L_{6,24}(\epsilon)$.
		
		\item[(ii)] $s(L)=7$ if and only if $L$ is isomorphic to one of the Lie algebras 
		$L_{5,8}\oplus A(6)$, $L_{4,3}\oplus A(5)$,	$L_{5,5}\oplus A(4)$, $L_{6,22}(\epsilon)\oplus A(4)$,  $27B$, $L_{6,10}\oplus A(2)$, $27A \oplus A(1)$,  $157 \oplus A(1)$, $L_{6,10} \dotplus H(1)$,  $H(1)\oplus H(2)$, $S_1$, $S_2$, $S_3$, $L_{6,23} \oplus A(1)$,  $L_{6,25} \oplus A(1)$, $37B \oplus A(1)$, $37C \oplus A(1)$, $37D \oplus A(1)$, $L_{6,26} \oplus A(2)$, $L_{6,13}$, $257A$,  $257C$, $257F$ or $L_{6,21}(\epsilon)$.		
	\end{itemize}	
\end{theorem}
Here, the notations used for nilpotent Lie algebras are taken from \cite{sc.wg,mpg}. 
It is necessary to mention that $L_{6,i}=L_{5,i} \oplus A(1)$, for all $i \leq 9$.

The following theorems are needed for proving our results and we state them without proofs.
\begin{theorem}[\cite{pn.fr}]\label{direct sum of multiplier}
Let $L$ and $K$ be two finite dimensional Lie algebras. Then
\[\dim\M(L\oplus K)=\dim\M(L)+\dim\M(K)+\dim (L/L^2)\dim (K/K^2).\]
\end{theorem}
\begin{theorem}[\cite{pn.fr}]\label{dim Heisenberg}
Let $H(m)$ be the Heisenberg Lie algebra of dimension $2m+1$ with $m\geq1$. Then $\dim\M(H(1))=2$ and $\dim\M(H(m))=2m^2-m-1$, for $m\geq2$.
\end{theorem}
\begin{theorem}[\cite{hd.fs.me.1}]\label{dimA^2=1}
	Let $L$ be a $d$-dimensional nilpotent Lie algebra  with $\dim L^2=1$. Then $L\cong H(m)\oplus A(d-2m-1),$ for some $m\geq1$. Moreover, 
	\[\dim\M(L)=\begin{cases}\frac{1}{2}(d-1)(d-2)+1,&m=1,\\\frac{1}{2}(d-1)(d-2)-1,&m\geq2.\end{cases}\]
\end{theorem}
\section{Main results}

In this section, we classify all pairs of finite dimensional nilpotent Lie algebras $(N,L)$ with $s(N,L)=6, 7$.

Let $L$ be a finite dimensional nilpotent Lie algebra and $N$, $K$ be ideals of $L$ such that $L=N \oplus K$. Let $\dim N=n$ , $\dim K=m$ and $\dim N^2=d \geq 1$.
 We have $\dim \M(L)=\dim \M(N) + \dim \M(K) +(n-\dim N^2)(m-\dim K^2),$ by Theorem \ref{direct sum of multiplier}. Also  
$\dim \M(L)=\dim \M(N,L) + \dim \M(K)$
(see \cite{ha.fs.mr.ek} and \cite{fs.as.be}). Thus 
 \begin{equation}\label{equality}
	\dim \M(N,L)=\dim \M(N) +(n-\dim N^2)(m-\dim K^2).
\end{equation}

By using Theorem 3.1 of \cite{pn.fr},  $ \dim \M(N) \leq \frac{1}{2} (n+\dim N^2-2)(n-\dim N^2-1)+1$.
Therefore 
 $$ \dim \M(N,L)\leq \frac{1}{2} (n+\dim N^2-2)(n-\dim N^2-1)+1+m(n-\dim N^2).$$
As the above bound for $\dim \M(N,L)$  is decreasing with respect to $\dim N^2$, there exists an integer $s(N,L) \geq 0$ in such a way that
 $$ \dim \M(N,L)=\frac{1}{2} (n-1)(n-2)+1+m(n-1)-s(N,L).$$

By replacing $\dim \M(N)= \frac{1}{2}(n-1)(n-2)+1-s(N)$ in (\ref*{equality}), we have

 \begin{equation}\label{s1}
	(n-1)m -(n-\dim N^2)(m-\dim K^2)=s(N,L)-s(N).
\end{equation}
It is easy to see that

\begin{equation}\label{s2}
	s(N,L)-s(N) \geq m(\dim N^2-1),
\end{equation}
which implies that $s(N,L) \geq s(N)$.

The following theorem determines the structure of all pairs of  nilpotent Lie algebras $(N,L)$, when $s(N,L) \leq 2$. 

\begin{theorem} \cite{ek.sa}
	Let $(N,L)$ be a pair of finite dimensional non-abelian nilpotent Lie algebras and $K$ be an ideal of $L$ such that $L=N \oplus K$, $\dim N=n$, $\dim K=m$ and  $\dim \M(N,L)= \frac{1}{2}(n-1)(n-2)+1+(n-1)m-s(N,L)$. Then
	\begin{itemize}		
		\item[(i)]$s(N,L)=0$ if and only if $(N,L)\cong (H(1)\oplus A(n-3),H(1)\oplus A(n+m-3))$.
		\item[(ii)]$s(N,L)=1$ if and only if $(N,L)\cong (L_{5,8},L_{5,8})$. 
		\item[(iii)]$s(N,L)=2$ if and only if $(N,L)$ is isomorphic to 
    	$(H(1),H(1)\oplus H(r)\oplus A(m-2r-1)),  r\geq1$,
    	$(H(r)\oplus A(n-2r-1),H(r) \oplus A(m+n-2r-1))$,
    	$(L_{5,8} \oplus A(i),L_{5,8} \oplus A(1)), i=0,1$
    	or $(L_{4,3},L_{4,3})$.
	\end{itemize}	
\end{theorem}

Now, we are ready to state and prove our main results.
\begin{theorem} \label{s(N,L)=6}
	Let $(N,L)$ be a pair of finite dimensional non-abelian nilpotent Lie algebras and $K$ be a non-zero ideal of $L$ such that $L=N \oplus K$, $\dim N=n$, $\dim K=m$ and  $\dim \M(N,L)= \frac{1}{2}(n-1)(n-2)+1+(n-1)m-s(N,L)$. Then $s(N,L)=6$ if and only if $(N,L)$ is isomorphic to one of the following pairs of nilpotent Lie algebras:\\	
	 $(H(1)\oplus A(4),H(1)\oplus H(r) \oplus A(m-2r+3))$, for all  $r \geq 1$,\\  
     $(H(1)\oplus A(1),H(1)\oplus A(1) \oplus K)$, for all nilpotent Lie algebra $K$ with $\dim K^2=2$,\\ 		    
     $ (H(1),H(1)\oplus K)$, for all nilpotent Lie algebra $K$ with $\dim K^2=3$,\\ 		      
	   $(H(2) ,H(2) \oplus H(r) \oplus A(m-2r-1))$, for all $r \geq 1$,\\  		      
	    $(L_{5,8} \oplus A(i),L_{5,8} \oplus A(5))$, for all  $0 \leq i \leq 4$,\\  	    
    	$(L_{4,3}\oplus A(i),L_{4,3}\oplus A(4))$,  for all  $0 \leq i \leq 3$,\\      	
	 $(L_{5,5}\oplus A(1),L_{5,5}\oplus A(3))$,  for all $0 \leq i \leq 2$,\\  	 
		$(L_{6,22}(\epsilon)\oplus A(i),L_{6,22}(\epsilon)\oplus A(3))$, for all $0 \leq i \leq 2$,\\  		
	$(L_{5,j},L_{5,j} \oplus A(1)),$  for $j=6,7,8$,\\ 		
 $(37A,37A\oplus A(1))$,\\ 	
$(L_{6,10}, L_{6,10} \oplus A(1))$.	
\end{theorem}
\begin{proof}
	
   Let $s(N,L)=6$. By using (\ref{s2}), $s(N)$ is equal to $0$, $1$, $2$, $3$, $4$, $5$ or $6$.
   First, assume that $s=0$. In this case, by Theorem \ref{s=0,1,2,3}, $N=H(1) \oplus A(n-3)$ and $\dim N^2=1$.
   From the relation (\ref{s1}), we have $(n-1)\dim K^2=6$. As $n \geq 3$, we have $n=3,4$ or $7$. If $n=7$, then $\dim K^2=1$. Theorem \ref{dimA^2=1} shows that $K \cong H(r) \oplus A(m-2r-1),$ for some $r\geq 1$. Therefore, $(N,L)\cong (H(1)\oplus A(4),H(1)\oplus H(r) \oplus A(m-2r+3)),  r \geq 1$. If $n=4$, then $\dim K^2=2$. Thus $(N,L)\cong (H(1)\oplus A(1), H(1)\oplus A(1) \oplus K)$. If, $n=3$, then $\dim K^2=3$. Hence $(N,L)\cong (H(1),H(1)\oplus K)$.
   
  
   Now, let $s(N)=1$, then by Theorem \ref{s=0,1,2,3}, $N=L_{5,8}$. In this case, $n=5$, $\dim N^2=2$ and hence, $m=5-3\dim K^2$, by (\ref{s1}). On the other hand, the only acceptable value for $m$ is $5$. Therefore, $K=A(5)$ and $(N,L) \cong (L_{5,8},L_{5,8}\oplus A(5))$.
   
  
   Assume that $s(N)=2$. By Theorem \ref{s=0,1,2,3}, $N \cong L_{4,3}$, $L_{5,8} \oplus A(1)$ or $H(r)\oplus A(n-2r-1),  r\geq2$. If $N$ is isomorphic to $L_{4,3}$, then $m=4-2\dim K^2$, by (\ref{s1}). But the only acceptable value for $m$ is $4$. So, $K=A(4)$ and $(N,L)$ is isomorphic to $(L_{4,3},L_{4,3}\oplus A(4))$.   
   If $N$ is isomorphic to $L_{5,8} \oplus A(1)$, then $m=4-4\dim K^2$, by (\ref{s1}). The only acceptable value for $m$ is $4$. Hence $K=A(4)$ and $(N,L)$ is isomorphic to $(L_{5,8} \oplus A(1),L_{5,8} \oplus A(5))$.
   If $N$ is isomorphic to $H(r)\oplus A(n-2r-1),  r\geq2$, then $(n-1)\dim K^2=4$, by (\ref{s1}). Since $n \geq 5$, the only acceptable values for $n$ and $\dim K^2$ are $5$ and $1$, respectively. Hence $N=H(2)$ and $K=H(r) \oplus A(m-2r-1),  r \geq 1$. Thus  $(N,L)$ is isomorphic to $(H(2) ,H(2) \oplus H(r) \oplus A(m-2r-1)), r \geq 1$.
   
    
   Suppose that $s(N)=3$, then $N \cong L_{4,3}\oplus A(1), L_{5,5}, L_{6,22}(\epsilon), L_{6,26}$ or $L_{5,8} \oplus A(2)$,  by Theorem \ref{s=0,1,2,3}. If $N$ is isomorphic to $L_{4,3}\oplus A(1)$ or $L_{5,5}$, then $m=3-3\dim K^2$, by (\ref{s1}). But the only acceptable value for $m$ is $3$. Therefore $K=A(3)$ and $(N,L)$ is isomorphic to $(L_{4,3}\oplus A(1),L_{4,3}\oplus A(4))$ or $(L_{5,5},L_{5,5}\oplus A(3))$. 
   If $N$ is isomorphic to $L_{6,22}(\epsilon)$ or $L_{5,8} \oplus A(2)$, then $m=3-4\dim K^2$ or $m=3-5\dim K^2$, respectively, by (\ref{s1}). The only acceptable value for $m$ is $2$ for both cases. Therefore $(N,L)$ is isomorphic to $(L_{6,22}(\epsilon),L_{6,22}(\epsilon)\oplus A(3))$ or $(L_{5,8}\oplus A(2),L_{5,8}\oplus A(5))$, respectively.  
   If $N$ is isomorphic to $L_{6,26}$, then $2m=3-3\dim K^2$, by (\ref{s1}). Thus in this case, we get a contradiction.
   
    
   Let $s(N)=4$. Then $N$ isomorphic to $L_{5,8}\oplus A(3)$, $L_{4,3}\oplus A(2)$, $L_{5,5}\oplus A(1)$, $L_{5,6}$, $L_{5,7}$,  $L_{5,9}$, $L_{6,22}(\epsilon)\oplus A(1)$ or $37A$, by Theorem \ref{s=4,5}.
   
   If $N$ is isomorphic to $L_{5,6}$, $L_{5,7}$ or $L_{5,9}$, then $m=1-\dim K^2$, by (\ref{s1}), and the only acceptable value for $m$ is $1$ for all three cases. Therefore, $(N,L)$ is isomorphic to $(L_{5,6},L_{5,6} \oplus A(1))$, $(L_{5,7},L_{5,7} \oplus A(1))$ or $(L_{5,9},L_{5,9} \oplus A(1))$. 
     
   If $N$ is isomorphic to $L_{4,3}\oplus A(2)$ or $L_{5,5}\oplus A(1)$, then $m=2-4\dim K^2$, by (\ref{s1}). But the only acceptable value for $m$ is $2$ for both cases. Hence $(N,L)$ is isomorphic to $(L_{4,3}\oplus A(2),L_{4,3}\oplus A(4))$ or $(L_{5,5}\oplus A(1),L_{5,5}\oplus A(3))$.
   If $N$ is isomorphic to $L_{5,8}\oplus A(3)$, $L_{6,22}(\epsilon)\oplus A(1)$ or $37A$, then $m=2-6\dim K^2$, $m=2-5\dim K^2$, $m=1-2\dim K^2$, respectively, by (\ref{s1}). One notes that the only acceptable value for $m$ is $2$, $2$ and $1$, respectively. Therefore, $(N,L)$ is isomorphic to $(L_{5,8}\oplus A(3), L_{5,8}\oplus A(5))$, $(L_{6,22}(\epsilon)\oplus A(1),L_{6,22}(\epsilon)\oplus A(3))$ or $(37A,37A\oplus A(1))$, respectively.
   Assume that $s(N)=5$. Then $N$ isomorphic to $L_{5,8}\oplus A(4)$, $L_{4,3}\oplus A(3)$, $L_{5,5}\oplus A(2)$, $L_{6,22}(\epsilon)\oplus A(2)$, $L_{6,26}\oplus A(1)$, $L_{6,10}$,  $L_{6,23}$, $L_{6,25}$, $L_{6,27}$, $37B$ or $37D$, by Theorem \ref{s=4,5}. If $\dim N^2=3$, then  we have a contradiction, by (\ref{s1}). Thus $\dim N^2=2$ and $K=A(1)$. In this case, $(N,L)$ is isomorphic to one of the following pairs.
   
   $(L_{5,8}\oplus A(4), L_{5,8}\oplus A(5))$, $(L_{4,3}\oplus A(3), L_{4,3}\oplus A(4))$, $(L_{5,5}\oplus A(2), L_{5,5}\oplus A(3))$, $(L_{6,22}(\epsilon)\oplus A(2), L_{6,22}(\epsilon)\oplus A(3))$,  $(L_{6,10}, L_{6,10} \oplus A(1))$.
   
   Finally, let $s(N)=6$. Then by using Theorem \ref{s=6,7}, $N$ isomorphic to $L_{5,8}\oplus A(5)$, $L_{4,3}\oplus A(4)$,	$L_{5,5}\oplus A(3)$, $L_{6,22}(\epsilon)\oplus A(3)$, 	$L_{6,10}\oplus A(1)$,  $27A$,  $157$,  $37A \oplus A(1)$, $L_{6,6}$,  $L_{6,7}$,  $L_{6,9}$, $L_{6,11}$, $L_{6,12}$,	$L_{6,19}(\epsilon)$, $L_{6,20}$ or $L_{6,24}(\epsilon)$. In this case, $\dim N^2 \geq 2$ and using the relation (\ref{s1}), we conclude that $L=N$.
   This completes the proof. 			 
\end{proof}

 \begin{theorem} \label{s(N,L)=7}
 	Let $(N,L)$ be a pair of finite dimensional non-abelian nilpotent Lie algebras and $K$ be a non-zero ideal of $L$ such that $L=N \oplus K$, $\dim N=n$, $\dim K=m$ and  $\dim \M(N,L)= \frac{1}{2}(n-1)(n-2)+1+(n-1)m-s(N,L)$. Then $s(N,L)=7$ if and only if $(N,L)$ is isomorphic to one of the following pairs of Lie algebras:\\
 	 $(H(1)\oplus A(5),H(1)\oplus H(r) \oplus A(m-2r+4))$,  for all $r \geq 1$,\\  	 
     $(H(2) \oplus A(1) ,H(2) \oplus H(r) \oplus A(m-2r))$, for all $r \geq 1$,\\ 
 	 $(L_{5,8} \oplus A(i),L_{5,8} \oplus A(6))$, for all $0 \leq i \leq 5$,\\   
     $(L_{4,3},L_{4,3}\oplus H(1))$,\\ 
 	 $(L_{4,3} \oplus A(i),L_{4,3} \oplus A(5))$, for all $0 \leq i \leq 4$,\\  
 	$(L_{6,22}(\epsilon)\oplus A(i),L_{6,22}(\epsilon)\oplus A(4))$, for all $0 \leq i \leq 3$,\\  
 	 $(L_{5,5}\oplus A(i),L_{5,5}\oplus A(4))$, for all $0 \leq i \leq 3$,\\ 
 	 $(L_{6,26}\oplus A(i), L_{6,26}\oplus A(2))$, for $i =0,1$,\\ 
 	  	$(L_{6,10}\oplus A(i),L_{6,10}\oplus A(2))$,  for $ i =0,1$,\\ 
 	$(L_{6,j}, L_{6,j}\oplus A(1))$, for $j=23,25,27$,\\   
 	 $(37B,37B \oplus A(1))$,\\ 	 
 	    $(37D,37D \oplus A(1))$,\\
  $(27A,27A \oplus A(1))$,\\
 $(157,157 \oplus A(1))$.			
 \end{theorem} 
     
  \begin{proof}  
By using (\ref{s2}) and assuming $s(N,L)=7$, we obtain $s(N)$ is equal to $0$, $1$, $2$, $3$, $4$, $5$, $6$ or $7$.
For the case $s=0$, Theorem \ref{s=0,1,2,3} implies that $N=H(1) \oplus A(n-3)$ and $\dim N^2=1$.
Using the relation (\ref{s1}), we conclude that $(n-1)\dim K^2=7$. As $n \geq 3$, we have $n=8$ and $\dim K^2=1$. Applying Theorem \ref{dimA^2=1}, we obtain that $K \cong H(r) \oplus A(m-2r-1)$, for some $r\geq 1$. Therefore, $(N,L)\cong (H(1)\oplus A(5),H(1)\oplus H(r) \oplus A(m-2r+4)), r \geq 1$.
Now, assume that $s(N)=1$. By Theorem \ref{s=0,1,2,3}, we have $N=L_{5,8}$. In this case, $n=5$, $\dim N^2=2$ and hence $m=6-3\dim K^2$, by (\ref{s1}). But the only acceptable values for $m$ are $6$ and $3$ and so $K=A(6)$ and $K=H(1)$, respectively. Thus $(N,L)$ is isomorphic to $(L_{5,8},L_{5,8}\oplus A(6))$ or $(L_{5,8},L_{5,8}\oplus H(1))$.
Let $s(N)=2$. Using Theorem \ref{s=0,1,2,3}, we conclude that $N \cong L_{4,3}, L_{5,8} \oplus A(1)$ or $H(r)\oplus A(n-2r-1), r\geq2$. If $N$ is isomorphic to $L_{4,3}$, then $m=5-2\dim K^2$ by (\ref{s1}). Note that the only acceptable values for $m$ are $5$ and $3$. Hence $K=A(5)$ and $K=H(1)$, respectively. Thus $(N,L)$ is isomorphic to $(L_{4,3},L_{4,3}\oplus A(5))$ or $(L_{4,3},L_{4,3}\oplus H(1))$. 

If $N$ is isomorphic to $L_{5,8} \oplus A(1)$, then $m=5-4\dim K^2$, by (\ref{s1}). The only acceptable value for $m$ is $5$. Therefore, $K=A(5)$ and $(N,L)$ is isomorphic to $(L_{5,8} \oplus A(1),L_{5,8} \oplus A(6))$.
If $N$ is isomorphic to $H(r)\oplus A(n-2r-1), r\geq2$, then $(n-1)\dim K^2 =5$, by (\ref{s1}). Since $n \geq 5$, the only acceptable values for $n$ and $\dim K^2$ are $n=6$ and $\dim K^2=1$. Hence $N=H(2)\oplus A(1)$ and $K=H(r) \oplus A(m-2r-1), r \geq 1$. Thus  $(N,L)$ is isomorphic to $(H(2) \oplus A(1)$ or $H(2) \oplus H(r) \oplus A(m-2r)), r\geq 1$.
Assume that $s(N)=3$, then $N \cong L_{4,3}\oplus A(1), L_{5,5}, L_{6,22}(\epsilon), L_{6,26}$ or $L_{5,8} \oplus A(2)$, by Theorem \ref{s=0,1,2,3}. If $N$ is isomorphic to $L_{4,3}\oplus A(1)$ or $L_{5,5}$, then $m=4-3\dim K^2$, by (\ref{s1}). The only acceptable value for $m$ is $4$. Therefore, $K=A(4)$ and $(N,L)$ is isomorphic to $(L_{4,3}\oplus A(1),L_{4,3}\oplus A(5))$ or $(L_{5,5},L_{5,5}\oplus A(4))$.

If $N$ is isomorphic to $L_{6,22}(\epsilon)$, $L_{6,26}$ or $L_{5,8} \oplus A(2)$, then $m=4-4\dim K^2$,  $2m=4-3\dim K^2$  or $m=4-5\dim K^2$, respectively, by (\ref{s1}). But the only acceptable values for $m$ are $4$, $2$ and $4$, respectively. Therefore $(N,L)$ is isomorphic to $(L_{6,22}(\epsilon),L_{6,22}(\epsilon)\oplus A(4))$ or $(L_{6,26},L_{6,26}\oplus A(2))$ and $(L_{5,8}\oplus A(2),L_{5,8}\oplus A(6))$. 
Let $s(N)=4$. Then $N$ is isomorphic to $L_{5,8}\oplus A(3)$, $L_{4,3}\oplus A(2)$, $L_{5,5}\oplus A(1)$, $L_{5,6}$, $L_{5,7}$,  $L_{5,9}$, $L_{6,22}(\epsilon)\oplus A(1)$ or $37A$, by Theorem \ref{s=4,5}.

If $N$ is isomorphic to $L_{5,6}$, $L_{5,7}$ or $L_{5,9}$, then $2m=3-2\dim K^2$, by (\ref{s1}). In this case,  we have a contradiction.
If $N$ is isomorphic to $L_{4,3}\oplus A(2)$ or $L_{5,5}\oplus A(1)$, then $m=3-4\dim K^2$, by (\ref{s1}). The only acceptable value for $m$ is $3$ for both cases. Therefore, $(N,L)$ is isomorphic to $(L_{4,3}\oplus A(2),L_{4,3}\oplus A(5))$ or $(L_{5,5}\oplus A(1),L_{5,5}\oplus A(4))$.
If $N$ is isomorphic to $L_{5,8}\oplus A(3)$ or $L_{6,22}(\epsilon)\oplus A(1)$, then $m=3-6\dim K^2$, $m=3-5\dim K^2$, respectively, by (\ref{s1}). On the other hand, the only acceptable value for $m$ is $3$ for both cases. Therefore, $(N,L)$ is isomorphic to $(L_{5,8}\oplus A(3), L_{5,8}\oplus A(6)), $ $(L_{6,22}(\epsilon)\oplus A(1),L_{6,22}(\epsilon)\oplus A(4))$, respectively.
If $N$ is isomorphic to $37A$, then $2m=3-4\dim K^2$, by (\ref{s1}). In this case,  we have a contradiction.

If $s(N)=5$, then $N$ is isomorphic to $L_{5,8}\oplus A(4)$, $L_{4,3}\oplus A(3)$, $L_{5,5}\oplus A(2)$, $L_{6,22}(\epsilon)\oplus A(2)$, $L_{6,26}\oplus A(1)$, $L_{6,10}$,  $L_{6,23}$, $L_{6,25}$, $L_{6,27}$, $37B$ or $37D$, by Theorem \ref{s=4,5}.
If $N$ is isomorphic to $L_{5,8}\oplus A(4)$, $L_{6,10}$ or $L_{6,22}(\epsilon)\oplus A(2)$, then $m=2-7\dim K^2$, $m=2-6\dim K^2$,  $m=2-4\dim K^2$, respectively, by (\ref{s1}). The only acceptable value for $m$ is $2$ for all three cases. Therefore $(N,L)$ is isomorphic to $(L_{5,8}\oplus A(4), L_{5,8}\oplus A(6))$, $L_{6,10}, L_{6,10} \oplus A(2)$,  $(L_{6,22}(\epsilon)\oplus A(2),L_{6,22}(\epsilon)\oplus A(4))$, respectively.
If $N$ is isomorphic to $L_{4,3}\oplus A(3)$ or $L_{5,5}\oplus A(2)$, then $m=2-5\dim K^2$, by (\ref{s1}). In this case, $(N,L)$ is isomorphic to $(L_{4,3}\oplus A(3), L_{4,3}\oplus A(5))$ or  $(L_{5,5}\oplus A(2),L_{5,5}\oplus A(4))$.
If $N$ is isomorphic to $L_{6,26}\oplus A(1)$, $37B$ or $37D$, then $m=1-2\dim K^2$, by (\ref{s1}) and hence $(N,L)$ is isomorphic to $(L_{6,26}\oplus A(1), L_{6,26}\oplus A(2))$, $(37B,37B \oplus A(1))$ or   $(37D,37D \oplus A(1))$.
If $N$ is isomorphic to $L_{6,23}$, $L_{6,25}$ or $L_{6,27}$, then $2m=2-3\dim K^2$, by (\ref{s1}). But the only acceptable value for $m$ is $1$ for all three cases. In this case, $(N,L)$ is isomorphic to $(L_{6,23}, L_{6,23}\oplus A(1))$, $(L_{6,25}, L_{6,25}\oplus A(1))$ or  $(L_{6,27}, L_{6,27}\oplus A(1))$.
Suppose that $s(N)=6$. Then $N$ is isomorphic to $L_{5,8}\oplus A(5)$, $L_{4,3}\oplus A(4)$,	$L_{5,5}\oplus A(3)$, $L_{6,22}(\epsilon)\oplus A(3)$,  $L_{6,10}\oplus A(1)$,  $27A$,  $157$,  $37A \oplus A(1)$, $L_{6,6}$,  $L_{6,7}$,  $L_{6,9}$, $L_{6,11}$, $L_{6,12}$,	$L_{6,19}(\epsilon)$, $L_{6,20}$ or $L_{6,24}(\epsilon)$, by Theorem \ref{s=6,7}.

If $\dim N^2=3$ then  we obtain a contradiction, by (\ref{s1}). Thus $\dim N^2=2$ and $K=A(1)$. In this case, $(N,L)$ is isomorphic to one of the following pairs.

$(L_{5,8}\oplus A(5),L_{5,8}\oplus A(6))$, $(L_{4,3}\oplus A(4),L_{4,3}\oplus A(5))$,	$(L_{5,5}\oplus A(3),L_{5,5}\oplus A(4))$, $(L_{6,22}(\epsilon)\oplus A(3),L_{6,22}(\epsilon)\oplus A(4))$, 	$(L_{6,10}\oplus A(1),L_{6,10}\oplus A(2))$,  $(27A,27A \oplus A(1))$ or  $(157,157 \oplus A(1))$.

Finally, If $s(N)=7$, then by using Theorem \ref{s=6,7}, $N$ isomorphic to 	$L_{5,8}\oplus A(6)$, $L_{4,3}\oplus A(5)$,	$L_{5,5}\oplus A(4)$, $L_{6,22}(\epsilon)\oplus A(4)$,  $27B$, $L_{6,10}\oplus A(2)$, $27A \oplus A(1)$,  $157 \oplus A(1)$, $L_{6,10} \dotplus H(1)$,  $H(1)\oplus H(2)$, $S_1$, $S_2$, $S_3$, $L_{6,23} \oplus A(1)$,  $L_{6,25} \oplus A(1)$, $37B \oplus A(1)$, $37C \oplus A(1)$, $37D \oplus A(1)$, $L_{6,26} \oplus A(2)$, $L_{6,13}$, $257A$,  $257C$, $257F$ or $L_{6,21}(\epsilon)$.	
In this case, $\dim N^2 \geq 2$. Using the relation (\ref{s1}), we conclude that $L=N$, 
which completes the proof. 						
\end{proof}

In the following theorem, we classify all pairs of nilpotent Lie algebras $(N,L)$ with $s(N,L)=3, 4, 5$. Due to the similarity of the proofs with the ones of Theorems \ref{s(N,L)=6} and \ref{s(N,L)=7}, we omit the proofs.

\begin{theorem}
	Let $(N,L)$ be a pair of finite dimensional non-abelian nilpotent Lie algebras and $K$ be a non-zero ideal of $L$ such that $L=N \oplus K$, $\dim N=n$, $\dim K=m$ and  $\dim \M(N,L)= \frac{1}{2}(n-1)(n-2)+1+(n-1)m-s(N,L)$. Then
	\begin{itemize}	
		
		\item[(i)]$s(N,L)=3$ if and only if $(N,L)$ is isomorphic to
		$(H(1)\oplus A(1),H(1)\oplus H(r) \oplus A(m-2r)), r \geq 1$, $(L_{5,8} \oplus A(i),L_{5,8} \oplus A(2)), 0 \leq i \leq 2$, or	$(L_{4,3}\oplus A(i),L_{4,3}\oplus A(1)), i=0,1$.
		
		\item[(ii)]$s(N,L)=4$ if and only if $(N,L)$ is isomorphic to one of the following pairs of Lie algebras:\\
		$(H(1)\oplus A(2),H(1)\oplus H(r) \oplus A(m-2r+1))$, for all  $r \geq 1$,\\
		$(L_{5,8} \oplus A(i),L_{5,8} \oplus A(3))$, for all  $0 \leq i \leq 3$,\\
		$(L_{4,3}\oplus A(i),L_{4,3}\oplus A(2))$, for all $0 \leq i \leq 2$,\\
		$(L_{5,5}\oplus A(i),L_{5,5}\oplus A(1))$, for $ i = 0, 1$,\\
		$(L_{6,22}(\epsilon)\oplus A(i),L_{6,22}(\epsilon)\oplus A(1))$, for $i = 0, 1$.
		
		\item[(iii)]$s(N,L)=5$ if and only if $(N,L)$ is isomorphic to one of the following pairs of Lie algebras:\\
		$(H(1)\oplus A(3),H(1)\oplus H(r)\oplus A(m-2r+2))$, for all $r \geq 1$,\\
		$(L_{5,8} \oplus A(i),L_{5,8} \oplus A(4))$, for all  $0 \leq i \leq 4$,\\	$(L_{4,3}\oplus A(i),L_{4,3}\oplus A(3))$, for all $0 \leq i \leq 3$,\\
		$(L_{5,5}\oplus A(i),L_{5,5}\oplus A(2))$, for all  $0 \leq i \leq 2$,\\
		$(L_{6,22}(\epsilon)\oplus A(i),L_{6,22}(\epsilon)\oplus A(2))$, for all $0 \leq i \leq 2$,\\
		$(L_{6,26}\oplus A(i),L_{6,26}\oplus A(1))$, for $ i= 0, 1$.	
	\end{itemize}	
\end{theorem}

For more information of readers, the classification of all Lie algebras, which are used in this paper have been collected in the following Table 1.

\begin{longtable} {|p{2.1cm}|p{9.8cm}|}
	\caption{}
	\label{tab:table3}\\
	\hline
	\textbf{Lie algebras} & \textbf{Non-zero Lie multiplications}  \\
	\hline
	$L_{4,3}$&$[e_1,e_2]=e_3$,\ $[e_1,e_3]=e_4$\\
	\hline
	$L_{5,5}$&$[e_1,e_2]=e_3$,\ $[e_1,e_3]=[e_2,e_4]=e_5$\\
	\hline
	$L_{5,6} $& $ [e_1,e_2]=e_3, [e_1,e_3]=e_4, [e_1,e_4]=[e_2,e_3] = e_5$ \\
	\hline 
	$L_{5,7}$&$ [e_1,e_2]=e_3, [e_1,e_3]=e_4, [e_1,e_4]=e_5$\\   
	\hline 
	$L_{5,8}$&$[e_1,e_2]=e_4$,\ $[e_1,e_3]=e_5$\\
	\hline 
	$L_{5,9}$&$[e_1,e_2]=e_3, [e_1,e_3]=e_4, [e_2,e_3]=e_5$\\
	\hline 
	$L_{6,10} $& $ [e_1,e_2]=e_3, [e_1,e_3]= [e_4,e_5]=e_6$\\
	\hline 
	$L_{6,11}$&$[e_1,e_2]=e_3,[e_1,e_3]=e_4, [e_1,e_4]=[e_2,e_3]=[e_2,e_5]=e_6$\\
	\hline 
	$L_{6,12} $&$ [e_1,e_2]=e_3,[e_1,e_3]=e_4, [e_1,e_4]= [e_2,e_5]=e_6$ \\		
	\hline 
	$L_{6,13} $&$ [e_1,e_2]=e_3,[e_1,e_3]=[e_2,e_4]=e_5, [e_1,e_5]=[e_3,e_4]=e_6$ \\		
	\hline 
	$L_{6,20} $& $[e_1,e_2]=e_4, [e_1,e_3]=e_5, [e_1,e_5] =[e_2,e_4]=e_6$ \\
	\hline 
	$L_{6,23} $& $ [e_1,e_2]=e_3, [e_1,e_3]=[e_2,e_4] = e_5, [e_1,e_4]=e_6 $ \\
	\hline 
	$L_{6,25} $&$ [e_1,e_2]=e_3, [e_1,e_3]=e_5, [e_1,e_4]=e_6$ \\
	\hline 
	$L_{6,26} $&$ [e_1,e_2]=e_4, [e_1,e_3]=e_5, [e_2,e_4]=e_6$ \\
	\hline 
	$L_{6,27} $& $ [e_1,e_2]=e_3, [e_1,e_3]=e_5, [e_2,e_4]=e_6$\\ 
	\hline 
	$L_{6,19}(\epsilon) $&$[e_1,e_2]=e_4,[e_1,e_3]=e_5,[e_1,e_5]=[e_2,e_4]=e_6,[e_3,e_5]=\epsilon e_6$\\		
	\hline
	$L_{6,21}(\epsilon) $&$[e_1,e_2]=e_3,[e_1,e_3]=e_4,[e_2,e_3]=e_5, [e_1,e_4]=e_6,[e_2,e_5]=\epsilon e_6$\\			
	\hline 
	$L_{6,22}(\epsilon) $& $ [e_1,e_2]=[e_3,e_4]=e_5, [e_1,e_3]= e_6 , [e_2,e_4]=\epsilon e_6 $ \\	
	\hline
	$L_{6,24}(\epsilon) $&$ [e_1,e_2]=e_3,[e_1,e_3]=[e_2,e_4]=e_5,[e_2,e_3]=e_6,[e_1,e_4]=\epsilon e_6$\\	
	\hline
	$37A$&$[e_1,e_2]=e_5,[e_2,e_3]=e_6,[e_2,e_4]=e_7$\\			
	\hline
	$37B$&$[e_1,e_2]=e_5,[e_2,e_3]=e_6,[e_3,e_4]=e_7$\\		
	\hline
	$37C$&$[e_1,e_2]=[e_3,e_4]=e_5,[e_2,e_3]=e_6,[e_2,e_4]=e_7$\\		
	\hline
	$37D$&$[e_1,e_2]=[e_3,e_4]=e_5,[e_1,e_3]=e_6,[e_2,e_4]=e_7$\\				
	\hline
	$27A$&$[e_1,e_2]=e_6,[e_1,e_4]=[e_3,e_5]=e_7$\\	
	\hline	
	$27B$&$[e_1,e_2]=[e_3,e_4]=e_6, [e_1,e_5]=[e_2,e_3]=e_7$\\	
	\hline	
	$157$&$[e_1,e_2]=e_3,[e_1,e_3]=[e_2,e_4]=[e_5,e_6]=e_7$\\	
	\hline	
	$257A$&$[e_1,e_2]=e_3,[e_1,e_3]=[e_2,e_4]=e_6, [e_1,e_5]=e_7$\\	
	\hline	
	$257C$&$[e_1,e_2]=e_3,[e_1,e_3]=[e_2,e_4]=e_6, [e_2,e_5]=e_7$\\	
	\hline	
	$257F$&$[e_1,e_2]=e_3,[e_2,e_3]=[e_4,e_5]=e_6, [e_2,e_4]=e_7$\\	
	\hline	
	$S_1$&$[e_1,e_2]=e_6,[e_1,e_4]=[e_3,e_5]=[e_2,e_7]=e_8$\\	
	\hline	
	$S_2$&$[e_1,e_2]=e_4,[e_1,e_3]=[e_6,e_7]=[e_7,e_8]=e_5$\\	
	\hline	
	$S_3$&$[e_1,e_3]=e_6,[e_1,e_2]=[e_3,e_4]=[e_7,e_8]=e_5$\\	
	\hline
\end{longtable}

\end{document}